\documentclass[reqno]{amsart}
\usepackage{amssymb}
\usepackage{hyperref}
\usepackage{booktabs}
\usepackage{multirow}
\usepackage{tikz}
\usepackage{mathdots}

\newtheorem{theorem}{Theorem}
\newtheorem{proposition}[theorem]{Proposition}
\newtheorem{lemma}[theorem]{Lemma}
\newtheorem{corollary}[theorem]{Corollary}

\theoremstyle{remark}

\numberwithin{equation}{section}

\begin{document}

\title[Elliptic Racah polynomials]
{Elliptic Racah polynomials}

\author{Jan Felipe  van Diejen}

\address{
Instituto de Matem\'aticas, Universidad de Talca,
Casilla 747, Talca, Chile}

\email{diejen@inst-mat.utalca.cl}

\author{Tam\'as G\"orbe}

\address{School of Mathematics, University of Leeds, Leeds LS2 9JT,
UK}

\email{T.Gorbe@leeds.ac.uk}


\date{March 2021}

\begin{abstract} 
Upon solving a finite discrete reduction of the difference Heun equation, we arrive at an elliptic generalization of the Racah polynomials.
We exhibit the three-term recurrence relation and the orthogonality relations for these elliptic Racah polynomials.
The well-known $q$-Racah polynomials of  Askey and Wilson are recovered as a trigonometric limit.
\end{abstract}

\maketitle

\section{Introduction}\label{sec1}
The  Askey-Wilson polynomials
\cite{ask-wil:some} constitute
a master family from which all other members listed in Askey's celebrated scheme of (basic) hypergeometric orthogonal polynomials
 can be recovered via parameter specializations and limit transitions \cite{koe-les-swa:hypergeometric}. In particular,  for parameters subject to a suitable truncation condition the
Askey-Wilson polynomials reduce to $q$-Racah polynomials
\cite{ask-wil:set},  a  finite-dimensional discrete orthogonal family that is known to express the
$6j$ symbols associated with the  $SL_q(2)$ quantum group \cite{kir-res:representations}.  In the limit $q\to 1$, this reproduces a previously observed interpretation of the classical $6j$ symbols for the Lie group $SL(2)$ in terms of a hypergeometric orthogonal family known as Racah polynomials, which arises similarly  as a finite discrete truncation of Wilson's master family of hypergeometric orthogonal polynomials.

A remarkable elliptic hypergeometric generalization of the $6j$ symbols originating from the Yang-Baxter equation for exactly solvable lattice models
\cite{dat-jim-miw-oka:fusion,dat-jim-kun-miw-oka:exactly} has been identified and studied by Frenkel and Turaev \cite{fre-tur:elliptic}. It was pointed out by Spiridonov and Zhedanov \cite{spi-zhe:spectral,spi-zhe:generalized}  that rather than expressing orthogonal polynomials, these elliptic $6j$ symbols constitute in fact
an elliptic hypergeometric counterpart of biorthogonal rational functions that had been found previously at the basic hypergeometric level by Wilson
as a (non-polynomial) generalization of the $q$-Racah polynomials \cite{wil:orthogonal}.   From the point of view of representation theory, the elliptic hypergeometric biorthogonal rational functions in question can be seen, respectively, as $6j$ symbols for the elliptic quantum group associated with $U(2)$ \cite{koe-nor-ros:elliptic}
or as   $6j$ symbols  for the Sklyanin algebra \cite{ros:elementary,skl:some}.
A corresponding extension of the Askey scheme to the case of (basic) hypergeometric
biorthogonal rational functions has been worked out in \cite{bul-rai:basic,bul-rai:limits}.

The hallmark duality symmetry \cite{ask-wil:set,leo:orthogonal} between the orthogonality relations and the dual orthogonality relations for the ($q$-)Racah polynomials and the corresponding $6j$ symbols is known to persist at the level of the elliptic hypergeometric biorthogonal rational functions and the elliptic $6j$ symbols
\cite{spi-zhe:spectral} (cf. also \cite{koe-nor-ros:elliptic}).
The purpose of the present note, however, is to point out an elliptic generalization of the ($q$-)Racah orthogonal polynomials that avoids the transition to biorthogonal rational functions, at the expense of sacrificing this
manifest duality symmetry.
To this end we start from
a difference Heun equation that is obtained from the eigenvalue problem for a quantum
 Ruijsenaars-Schneider type particle Hamiltonian introduced in \cite{die:integrability} (cf. also \cite{kom-hik:quantum,kom-hik:conserved} for a proof of the integrability), upon specializing to the case of just a single particle. 
Systematic studies of the solutions of this difference Heun equation were performed in \cite{cha:bethe}  for integral values of the (coupling) parameters
and  in  \cite{rui:hilbert4} for parameters pertaining to a  much larger domain of orthogonality. 
Particular solutions for special parameter instances of the difference Heun equation can be found in
 \cite{tre:difference} (within the framework of the finite-gap integration of soliton equations) and in \cite{spi:elliptic,spi:continuous} (through elliptic hypergeometry).
Moreover, the difference Heun equation arises in the context of the representation theory of the Sklyanin algebra
\cite{ros:sklyanin,spi:continuous,rai-rui:difference}, as a linear problem associated with the elliptic Painlev\'e VI equation 
\cite{nou-rui-yam:elliptic}, and  it turns out to describe the introduction of
surface defects to the index computation of certain four dimensional compactifications of the six dimensional E string theory on a Riemann surface
\cite{naz-raz:surface}.

The difference Heun equation admits a rich hierarchy
of degenerations generalizing the Askey scheme (cf. \cite{die:difference}), the solutions of which are currently under active investigation
\cite{bas-tsu-vin-zhe:heun-askey,bas-vin-zhe:q-heun,ber-gab-vin-zhe:sklyanin,tak:degenerations,tak:q-deformations,tsu-vin-zhe:rational}. In this same spirit,
we will introduce below a finite-dimensional reduction of the difference Heun equation that is obtained by means of a truncation procedure that should be viewed as an elliptic counterpart  of the truncation yielding the $q$-Racah polynomials from the Askey-Wilson polynomials.
We thus end up with a finite discrete Heun equation describing the eigenvalue problem for a finite-dimensional tridiagonal matrix with explicit entries given by theta functions. By means of standard techniques from the theory of tridiagonal matrices, we will solve the corresponding spectral problem in terms of an orthogonal system of discrete Heun functions given by an elliptic generalization of the ($q$-)Racah polynomials. This elliptic Racah polynomial is defined by a tridiagonal determinant giving rise to a three-term recurrence relation and explicit orthogonality relations determined via the Christoffel-Darboux formula. In the trigonometric limit one recovers the $q$-Racah polynomial of Askey and Wilson.

Let us now outline the precise layout of this note.
In Section \ref{sec2} we recall the definition of the  difference Heun equation;  by implementing a truncation condition on the parameters
the finite discrete Heun equation is introduced.
This finite discrete Heun equation encodes the eigenvalue problem for a finite-dimensional tridiagonal matrix with simple spectrum. In
Section \ref{sec3} we construct the corresponding eigenvectors, which entails the elliptic Racah polynomials and their orthogonality relations.
By expanding the defining tridiagonal determinant an explicit formula for the elliptic Racah polynomials is obtained.
In Section \ref{sec4} it is verified that in the trigonometric limit, the elliptic Racah polynomials recuperate the $q$-Racah polynomials together with their recurrence relation and orthogonality relations. We also point out a  Lam\'e type  parameter reduction of the elliptic Racah polynomials that diagonalizes a recently found elliptic generalization of the Kac-Sylvester matrix \cite{die-gor:elliptic}. This is
a discrete elliptic counterpart of a well-known parameter reduction retrieving Rogers' $q$-ultraspherical  polynomials from the Askey-Wilson polynomials \cite{koe-les-swa:hypergeometric}.

\section{Finite discrete Heun equation}\label{sec2}

\subsection{Difference Heun equation}
The difference Heun equation  is an eigenvalue equation for a complex function $f(z)$:
\begin{subequations}
\begin{equation}\label{dH:eq}
\mathrm{H} f =\textsc{e} f ,
\end{equation}
which is determined by 
a linear second-order difference operator of the form
\begin{equation}\label{dH:op}
(\mathrm{H}f)(z)= \textsc{a}(z) f(z+1)+\textsc{a}(-z)f(z-1)+\textsc{b}(z)f(z) 
\end{equation}
and a  spectral parameter $\textsc{e}\in\mathbb{C}$; notice that we have
scaled the independent variable $z$ such that the steps of the difference equation take unit values. 
The coefficients $\textsc{a} (z)$ and $\textsc{b} (z)$ denote
meromorphic functions that are given explicitly by
\begin{equation}\label{coefA}
\textsc{a} (z)=\prod_{1\leq r\leq 4}{\textstyle \frac{[z+u_r]_r}{[z]_r}\frac{[z+\frac{1}{2}+v_r]_r}{[z+\frac{1}{2}]_r}}
\quad\text{and}\quad
\textsc{b} (z)=\sum_{1\leq r\leq 4} {\textstyle c_r\frac{[z+\frac{1}{2}+u]_r}{[z+\frac{1}{2}]_r}\frac{[z-\frac{1}{2}-u]_r}{[z-\frac{1}{2}]_r} } ,
\end{equation}
with
\begin{equation}\label{cr}
c_r={\textstyle \frac{2}{[u]_1[u+1]_1}  } \prod_{1\leq s\leq 4} {\textstyle [ u_{\pi_r(s)}-\frac{1}{2}]_s[v_{\pi_r(s)}]_s } .
\end{equation}
\end{subequations}
In these formulas $\pi_1,\ldots ,\pi_4$ stand for permutations that act on (the indices of) the parameters representing translations over the half-periods of the elliptic functions: $\pi_1=\mathrm{id}$, $\pi_2=(12)(34)$, $\pi_3=(13)(24)$,  $\pi_4=(14)(23)$.
Moreover, we have employed the following rescaled and normalized variants
\begin{equation}\label{scaled-thetafunctions}
[z]_1=\dfrac{\theta_1(\frac{\alpha}{2}z)}{\frac{\alpha}{2}\theta'_1(0)},\quad
[z]_2=\dfrac{\theta_2(\frac{\alpha}{2}z)}{\theta_2(0)},\quad
[z]_3=\dfrac{\theta_3(\frac{\alpha}{2}z)}{\theta_3(0)},\quad
[z]_4=\dfrac{\theta_4(\frac{\alpha}{2}z)}{\theta_4(0)},
\end{equation}
of the  Jacobi theta functions
\begin{equation*}
\begin{aligned}
\theta_1(z)=\theta_1(z;p)&=2\sum_{n=0}^\infty(-1)^np^{(n+\frac{1}{2})^2}\sin(2n+1)z\\
&=2p^{1/4}\sin(z)\prod_{n=1}^\infty(1-p^{2n})(1-2p^{2n}\cos(2z)+p^{4n}) ,\\
\theta_2(z)=\theta_2(z;p)&=2\sum_{n=0}^\infty p^{(n+\frac{1}{2} )^2}\cos(2n+1)z\\
&=2p^{1/4}\cos(z)\prod_{n=1}^\infty(1-p^{2n})(1+2p^{2n}\cos(2z)+p^{4n}) ,\\
\theta_3(z)=\theta_3(z;p)&=1+2\sum_{n=1}^\infty p^{n^2}\cos(2nz)\\
&=\prod_{n=1}^\infty(1-p^{2n})(1+2p^{2n-1}\cos(2z)+p^{4n-2}), \\
\theta_4(z)=\theta_4(z;p)&=1+2\sum_{n=1}^\infty(-1)^np^{n^2}\cos(2nz)\\
&=\prod_{n=1}^\infty(1-p^{2n})(1-2p^{2n-1}\cos(2z)+p^{4n-2}),
\end{aligned}
\end{equation*}
where $0<p<1$ stands for the elliptic nome and the scaling parameter $\alpha>0$ regulates the real period $\frac{2\pi}{\alpha}$  of the coefficients of $\mathrm{H}$.  The difference Heun equation
depends on eight coupling parameters $u_1,\ldots, u_4$, $v_1,\ldots ,v_4$ and a virtual regularization parameter $u$; this last parameter merely shifts the spectrum of
$\mathrm{H}$ (because the elliptic function $\textsc{b}(z)$ has only simple poles with positions and residues that do not depend on $u$).

The difference Heun operator $\mathrm{H}$ \eqref{dH:op}--\eqref{cr} goes back to a difference operator introduced in \cite[Eqs. (4.1)--(4.3)]{die:integrability} (upon specialization to the case $n=1$).
Notice that for a precise comparison between the above formulas and those in \cite{die:integrability} it is needed to pass from Jacobi theta functions to Weierstrass sigma functions associated with the period  lattice $\Omega=2\omega_1\mathbb{Z}+2\omega_2\mathbb{Z}$ (cf. e.g. \cite[Chapter 6]{law:elliptic}):
\begin{equation}\label{jw-conversion}
[z]_{r}= \sigma_{r-1} (z) e^{-\tfrac{\alpha\eta_1}{2\pi}z^2},\quad r=1,\ldots ,4, 
\end{equation}
where $\omega_1=\frac{\pi}{\alpha}$, $p=e^{i\pi \tau}$ with $\tau=\omega_3/\omega_1$, $\omega_3=-\omega_1-\omega_2$ and
\begin{equation*}
\sigma_0(z)=\sigma(z),\quad \sigma_s(z)=e^{-\eta_sz}\dfrac{\sigma(z+\omega_s)}{\sigma(\omega_s)} \quad\text{with}\   \eta_s=\zeta(\omega_s)\quad s=1,2,3.
\end{equation*}
Here $\sigma (z)$ and
$\zeta(z)=\sigma^\prime(z)/\sigma(z)$ stand for the Weierstrass sigma and zeta functions, respectively.  Indeed, it is readily seen by means
of the relation in Eq. \eqref{jw-conversion} that---upon conjugation with a Gaussian and multiplication by an overall constant---our difference Heun operator
$\mathrm{H}$ can be converted into a difference operator $\hat{\mathrm{H}}$ of the same form as in Eqs. \eqref{dH:op}--\eqref{cr} but with all rescaled theta functions $[\cdot]_r$ being replaced by sigma functions $\sigma_{r-1}(\cdot )$ ($r=1,\ldots ,4$): 
\begin{equation*}
\hat{\mathrm{H}}=  e^{-a+b} e^{-az^2}  \mathrm{H}e^{a z^2} ,
\end{equation*}
where $a=\frac{\alpha\eta_1}{2\pi} \sum_{1\leq r\leq 4}( u_r+v_r)$ and  $b=  \frac{\alpha\eta_1}{2\pi} \sum_{1\leq r\leq 4}( u_r^2+v_r^2 +v_r)$, i.e.
\begin{equation*}
    \mathrm{H}\to\hat{\mathrm{H}} \Longleftrightarrow [z]_r \to \sigma_{r-1}(z).
\end{equation*}
The  gauged and normalized difference Heun operator  $\hat{\mathrm{H}}$ thus obtained coincides therefore with the difference operator in  \cite[Eqs. (4.1)--(4.3)]{die:integrability}
(with $n=1$, $\beta\hbar=1/i$, $\gamma=1/2$,
$\mu=-u$, and $\mu_{r-1}=u_r$, $\mu_{r-1}^\prime=v_r$ for $r=1,\ldots ,4$).  

For the case of nonnegative integral values of the coupling parameters $u_1,\ldots,u_4$ and $v_1,\ldots,v_4$ eigenfunctions
for $\hat{\mathrm{H}}$ were computed in \cite{cha:bethe}. A more general construction of the difference Heun eigenfunctions
covering a much larger domain of parameter values can be found in \cite{rui:hilbert4}.

\subsection{Finite-dimensional reduction}
From now on we will pick real-valued coupling parameters $u_1,\ldots u_4$, $v_1,\ldots v_4$  from the domain
\begin{subequations}
\begin{equation}\label{pd:a}
\boxed{u_r  >0 ,\   |v_r| < u_r +{\textstyle \frac{1}{2}}    \ (r=1,2)\quad\text{and}\quad  u_r,v_r\in\mathbb{R}\ (r=3,4) ,}
\end{equation}
while throughout it will be assumed that the virtual parameter $u$ is chosen in $\mathbb{R}$ such that $u,u+1\neq 0\mod \frac{2\pi}{\alpha}\mathbb{Z}$. 
To truncate the difference Heun equation we adjust the real period $\frac{2\pi}{\alpha}$
in terms of the coupling parameters in the following way
\begin{equation}\label{pd:b}
\boxed{\alpha=\frac{\pi}{u_1+u_2+\textsc{m}}\quad\text{with}\ \textsc{m}\in\mathbb{N}}
\end{equation}
\end{subequations}
(so $u_1+u_2+\textsc{m}=\frac{\pi}{\alpha}$).
Indeed, the conditions on the parameters ensure that 
 the ($\textsc{m}+1$)-dimensional space of functions $f:\Lambda_{\textsc{m}}\to\mathbb{C}$ over the shifted
 finite integer lattice
\begin{equation*}
\Lambda_{\textsc{m}}=\{u_1,u_1+1,u_1+2,\ldots ,u_1+ \textsc{m} \} 
\end{equation*}
is stable for the action of the difference operator $\mathrm{H}$ \eqref{dH:op}--\eqref{cr}, because
$\textsc{a}(-u_1)=\textsc{a}(u_1+\textsc{m})=0$ in view of the zeros
of $[z]_1$ and $[z]_2$ at $z=0$ and $z=\frac{\pi}{\alpha}$, respectively.
This gives rise to the following finite-dimensional reduction of the difference Heun equation:
\begin{subequations}
\begin{equation}\label{dh:a}
\tilde{a}_{\textsc{m}-k} f_{k+1}+ a_k f_{k-1}+ b_k f_k=  \textsc{e} f_k
\end{equation}
for $ k=0,1,\ldots ,\textsc{m}$, where  $f_k=f(u_1+k)$ and
\begin{equation}\label{dh:b}
\tilde{a}_{k} =\textsc{a}(u_1+\textsc{m}-k),\quad a_k=\textsc{a}(-u_1-k),\quad b_k =\textsc{b}(u_1+k)
\end{equation}
\end{subequations}
(so $\tilde{a}_0=a_0=0$).

It is helpful to write out the coefficients in question explicitly:
\begin{subequations}
\begin{equation}\label{ak}
a_k = \prod_{1\leq r\leq 4}{\textstyle \frac{[u_1-u_r+k]_r}{[u_1+k]_r}\frac{[u_1-v_r-\frac{1}{2}+k ]_r}{[u_1-\frac{1}{2}+k]_r}}
\end{equation}
(because $[-z]_1=-[z]_1$ and $[-z]_r=[z]_r$ if $r\neq 1$),
\begin{equation}\label{tak}
\tilde{a}_k = \prod_{1\leq r\leq 4}{\textstyle \frac{[u_2-u_{\pi_2(r)}+k]_r}{[u_2+k]_r}\frac{[u_2-v_{\pi_2(r)}-\frac{1}{2}+k ]_r}{[u_2-\frac{1}{2}+k]_r}}
=\pi_2(a_k)
\end{equation}
(because $[z+\frac{\pi}{\alpha}]_r=[-z]_{\pi_2(r)}$),
and similarly
\begin{equation}
b_k = \sum_{1\leq r\leq 4} {\textstyle c_r\frac{[u_1+k+\frac{1}{2}+u]_r}{[z+\frac{1}{2}]_r}\frac{[u_1+k-\frac{1}{2}-u]_r}{[z-\frac{1}{2}]_r} } ,
\end{equation}
so
\begin{equation}\label{bk}
b_{\textsc{m}-k}= \sum_{1\leq r\leq 4} {\textstyle c_{\pi_2(r)}\frac{[u_2+k+\frac{1}{2}+u]_r}{[z+\frac{1}{2}]_r}\frac{[u_2+k-\frac{1}{2}-u]_r}{[z-\frac{1}{2}]_r} } =\pi_2(b_k)
\end{equation}
\end{subequations}
(because $\pi_r\circ \pi_s=\pi_{\pi_r(s)}$), where $\pi_r$ is understood to act on the coefficients $a_k$ and $b_k$ by permuting the parameters: $\pi_r(u_s)=u_{\pi_r(s)}$ and
$\pi_r(v_s)=v_{\pi_r(s)}$.  With the aid of these formulas one readily checks the positivity of the off-diagonal coefficients 
in the finite discrete Heun equation.

\begin{lemma}[Positivity]\label{positivity:lem}
The off-diagonal coefficients $a_1,\dots,a_{\textsc{m}}$ and $\tilde{a}_1,\dots,\tilde{a}_{\textsc{m}}$ of the finite discrete Heun equation
\eqref{dh:a}, \eqref{dh:b} are all positive.
\end{lemma}

\begin{proof}
Since $\tilde{a}_k=\pi_2(a_k)$ (cf. Eq. \eqref{tak}) and the parameter restrictions in Eqs. \eqref{pd:a}, \eqref{pd:b}  are invariant with respect to the action of $\pi_2$, it suffices to verify the positivity of $a_k$. To this end we observe from
the product expansions for the theta functions  that the sign of $a_k$ \eqref{ak} coincides with the overall sign of the
principal trigonometric factor:
\begin{equation*}
{\textstyle
\frac{\sin(\frac{\alpha}{2}k )}{\sin(\frac{\alpha}{2}(u_1+k))}\frac{\cos(\frac{\alpha}{2}(u_1-u_2+k))}{\cos(\frac{\alpha}{2}(u_1+k))}
\frac{\sin(\frac{\alpha}{2}(u_1-v_1-\frac{1}{2}+k)}{\sin(\frac{\alpha}{2}(u_1-\frac{1}{2}+k ))}\frac{\cos(\frac{\alpha}{2}(u_1-v_2-\frac{1}{2}+k))}{\cos(\frac{\alpha}{2}(u_1-\frac{1}{2}+k))} } .
\end{equation*}
For parameters in accordance with Eqs. \eqref{pd:a}, \eqref{pd:b} and $1\leq k\leq \textsc{m}$, the positivity of this trigonometric factor is clear
because all sine functions are evaluated at angles between
$0$ and $\pi$:
\begin{equation*}
{\textstyle 0<k<u_1+k<u_1+u_2+\textsc{m}=\frac{\pi}{\alpha}}
\end{equation*}
and
\begin{equation*}
{\textstyle 0<u_1-|v_1|-\frac{1}{2}+k\leq u_1-\frac{1}{2}+k \leq u_1+|v_1|-\frac{1}{2}+k <2u_1+\textsc{m} < 
\frac{2\pi}{\alpha} } ,
\end{equation*}
whereas all cosine functions are evaluated at angles between
$-\frac{\pi}{2}$ and $\frac{\pi}{2}$:
\begin{equation*}
-\frac{\pi}{\alpha} <u_1-u_2+k<u_1+k<\frac{\pi}{\alpha}
\end{equation*}
and
\begin{equation*}
{\textstyle -\frac{\pi}{\alpha}< u_1-u_2<   
u_1-|v_2|-\frac{1}{2}+k\leq u_1-\frac{1}{2}+k\leq u_1+|v_2|-\frac{1}{2}+k < \frac{\pi}{\alpha}}  .
\end{equation*}
\end{proof}

The upshot is that the finite discrete Heun equation \eqref{dh:a}, \eqref{dh:b} encodes the spectral problem for a real-valued finite-dimensional tridiagonal matrix of the form
\begin{subequations}
\begin{equation}\label{H:a}
\mathbf{H} \mathbf{f}=\textsc{e}\mathbf{f},
\end{equation}
with
\begin{equation}\label{H:b}
\mathbf{H}=\begin{bmatrix}
b_0 & \tilde{a}_{\textsc{m}}& 0&  \cdots & 0\\
a_1  &  b_1  &\ddots & & \vdots \\
0 & a_2 & \ddots &\tilde{a}_2&0\\
\vdots  &   &  \ddots   &b_{\textsc{m}-1}&\tilde{a}_1 \\
0 & \cdots & 0&a_{\textsc{m}} &b_{\textsc{m}}
\end{bmatrix}  
\quad\text{and}\quad
\mathbf{f}= \begin{bmatrix}
f_0 \\
f_1  \\
f_2 \\
\vdots\\  f_{\textsc{m}-1}  \\
f_{\textsc{m}} 
\end{bmatrix}  .
\end{equation}
\end{subequations}
In view of the positivity of the matrix elements on the sub- and superdiagonal by virtue of Lemma \ref{positivity:lem},  it is clear that
the spectrum of $\mathbf{H}$ is given by $\textsc{m}+1$ distinct and real eigenvalues
(cf. e.g.  \cite[Chapter III.11.4]{pra:problems}):
\begin{equation}\label{eigenvalues}
\textsc{e}_0>\textsc{e}_1>\dots>\textsc{e}_{\textsc{m}}.
\end{equation}
Moreover, the tridiagonal matrix  $\mathbf{H}$ in Eqs.  \eqref{H:a}, \eqref{H:b} is
quasi-centrosymmetric in the sense that its matrix elements $H_{j,k}$
obey the relation
\begin{equation}\label{qcs}
{H_{\textsc{m}-j,\textsc{m}-k}= \pi_2 \left(   H_{j,k}  \right) \quad \text{for}\ 0\leq j,k\leq \textsc{m}.}
\end{equation}

\section{Elliptic Racah polynomials}\label{sec3}

\subsection{Diagonalization}
Let $p_0(\textsc{e})=1$ and
\begin{equation}\label{e-racah}
p_{k}(\textsc{e})=\det 
\begin{bmatrix}
\textsc{e}-b_0 & -\tilde{a}_{\textsc{m}}& 0&  \cdots & 0\\
-a_1  &  \textsc{e}- b_1  &\ddots & & \vdots \\
0 & -a_2 & \ddots &-\tilde{a}_{\textsc{m}+3-k}&0\\
\vdots  &   &  \ddots   &\textsc{e}-b_{k-2}&-\tilde{a}_{\textsc{m}+2-k} \\
0 & \cdots & 0&-a_{k-1} &\textsc{e}-b_{k-1}
\end{bmatrix}  
\end{equation}
for $k=1,\ldots ,\textsc{m}+1$. In other words, $p_k(\textsc{e})$ is given by the $k$th
principal minor of the matrix $(\textsc{e}\mathbf{I}_{\textsc{m}+1}-\mathbf{H})$ governing the characteristic polynomial of $\mathbf{H}$
\eqref{H:b}. (Here $\mathbf{I}_{\textsc{m}+1}$ denotes the $(\textsc{m}+1)$-dimensional identity matrix.) We will refer to the polynomials
$p_0(\textsc{e}), p_1(\textsc{e}),\ldots ,p_{\textsc{m}+1}(\textsc{e})$ as (monic) \emph{elliptic Racah polynomials}.
By construction, these polynomials capture the characteristic polynomial of $\mathbf{H}$ at the top degree $k=\textsc{m}+1$:
\begin{equation}\label{charpol}
p_{\textsc{m}+1}(\textsc{e})=\det (\textsc{e}\mathbf{I}_{\textsc{m}+1}-\mathbf{H})= 
(\textsc{e}-\textsc{e}_0)(\textsc{e}-\textsc{e}_1)\cdots (\textsc{e}-\textsc{e}_\textsc{m})
\end{equation}
(cf. Eq. \eqref{eigenvalues}).

The following three-term recurrence relation
is manifest from the definition.

\begin{proposition}[Three-term Recurrence Relation]\label{rec:prp}
The monic elliptic Racah polynomials obey the three-term recurrence relation
\begin{equation}\label{rec}
p_{k+1}(\textsc{e})=(\textsc{e}-b_k)p_k(\textsc{e})-a_k\tilde{a}_{\textsc{m}+1-k}p_{k-1}(\textsc{e})\quad \text{for}\
k=0,\ldots,\textsc{m}.
\end{equation}
\end{proposition}
\begin{proof}
Immediate upon expanding the determinant for $p_{k+1}(\textsc{e})$ with respect to the last row/column.
\end{proof}

Moreover, by expanding the determinant $p_k(\textsc{e})$ \eqref{e-racah} as an alternating sum
of products of matrix elements pulled from the distinct rows/columns,
one arrives at the following explicit expression for the  polynomials in question.

\begin{proposition}[Elliptic Racah polynomials]\label{e-racah:prp}
The elliptic Racah polynomial of degree $0\leq k\leq\textsc{m}+1$ is given explicitly by
\begin{align}
&p_k(\textsc{e})=\\
&\sum_{l=0}^{\lfloor k/2\rfloor}(-1)^l
\sum_{\substack{1\leq j_1<j_2<\cdots<j_l< k \\  j_{s+1}-j_s>1\\ \text{for}\ s=1,\ldots ,l-1} } 
a_{j_1}\tilde{a}_{\textsc{m}+1-j_1}\cdots
a_{j_l}\tilde{a}_{\textsc{m}+1-j_l} \prod_{\substack{1\leq j\leq k\\j\not\in\{ j_s,j_s+1\}\\ \text{for}\ s=1,\ldots ,l}}(\textsc{e}-b_{j-1})  \nonumber
\end{align}
(with the convention that empty factors are equal to $1$).
\end{proposition}

\begin{proof}
Let us recall that the determinant of any $k\times k$ matrix
$[A_{i,j}]_{1\leq i,j \leq k }$ is given by an alternating sum of terms $(-1)^\tau A_{1,\tau(1)}A_{2,\tau(2)}\dots A_{k,\tau(k)}$
summed over all permutations $\tau= { \bigl( \begin{smallmatrix}1& 2& \cdots & k \\
 \tau (1)&\tau(2)&\cdots & \tau(k)
 \end{smallmatrix}\bigr)}$ of the symmetric group $S_k$ (where $(-1)^\tau$ refers to the sign of $\tau$).
 In the case of a triangular matrix,
 non-vanishing products can occur only when $\tau$ decomposes as a product of $0\leq l\leq \lfloor k/2\rfloor$ commuting simple transpositions:
\begin{equation*}
  \tau=  (j_1,j_1+1)(j_2,j_2+1)\cdots(j_l,j_l+1)
\end{equation*}
with
\begin{equation*}
1\leq j_1<j_1+1<j_2<j_2+1<\dots<j_l<j_l+1\leq k
\end{equation*}
(so the sign of $\tau$ is equal to $(-1)^l$). In the case of $p_k(\textsc{e})$ \eqref{e-racah}, each transposition
$(j_s,j_s+1)$ contributes a factor $a_{j_s}\tilde{a}_{\textsc{m}+1-j_s}$ to the product, while the indices $j$ that are fixed by $\tau$
each contribute a factor of the form $(\textsc{e}-b_{j-1})$. By collecting the contributions 
\begin{equation*}
(-1)^l a_{j_1}\tilde{a}_{\textsc{m}+1-j_1}\cdots
a_{j_l}\tilde{a}_{\textsc{m}+1-j_l} \prod_{\substack{1\leq j\leq k\\j\not\in\{ j_s,j_s+1\}\\ \text{for}\ s=1,\ldots ,l}}(\textsc{e}-b_{j-1})
\end{equation*}
from all such permutations $\tau$,
the asserted formula for $p_k(\textsc{e})$ follows.
\end{proof}

We are now in the position to solve the finite discrete Heun equation in terms of elliptic Racah polynomials.

\begin{proposition}[Eigenvectors]\label{evec:prp}
For any eigenvalue $\textsc{e}$ in the spectrum $\{ \textsc{e}_0>\textsc{e}_1>\cdots >\textsc{e}_{\textsc{m}}\}$ of
$\mathbf{H}$ \eqref{H:b} (i.e. on shell), the ($\textsc{m}+1$)-dimensional (column) vector $\mathbf{f}(\textsc{e})$ with components given by
normalized elliptic Racah polynomials of the form
\begin{equation}\label{evec}
f_k(\textsc{e})=p_k(\textsc{e})\prod_{0\leq j< k} \tilde{a}_{\textsc{m}-j}^{-1},\quad k=0,1,\dots,\textsc{m},
\end{equation}
solves the corresponding eigenvalue equation \eqref{H:a}.
\end{proposition}

\begin{proof}
Since on shell we assume that $\textsc{e}$ belongs to the spectrum of $\mathbf{H}$ and
$p_{\textsc{m}+1}(\textsc{e})=\det (\textsc{e}\mathbf{I}_{\textsc{m}+1}-\mathbf{H})$, it is clear
that $p_{\textsc{m}+1}(\textsc{e})=0$ in this situation. The three-term recurrence
relation \eqref{rec} for the elliptic Racah polynomials then affirms that on shell:
\begin{equation*}
\textsc{e}p_k(\textsc{e})=
\begin{cases}
p_{k+1}(\textsc{e}) +a_k\tilde{a}_{\textsc{m}+1-k}p_{k-1}(\textsc{e})+ b_k p_k(\textsc{e}) &
\text{for}\ k=0,\ldots,\textsc{m}-1,\\
a_\textsc{m}\tilde{a}_{1}p_{\textsc{m}-1}(\textsc{e})+ b_\textsc{m} p_\textsc{m}(\textsc{e}) &
\text{for}\ k=\textsc{m}.
\end{cases}
\end{equation*}
Multiplication of the $k$th equation by $\prod_{0\leq j< k} \tilde{a}_{\textsc{m}-j}^{-1}$ on both sides and rewriting the result
in terms of $f_k(\textsc{e})$ for $k=0,\ldots,\textsc{m}$, verifies that on shell the components of the vector $\mathbf{f}(\textsc{e})$ solve the finite discrete Heun equation
\eqref{dh:a}, \eqref{dh:b}.
\end{proof}

Since all eigenvalues \eqref{eigenvalues} are simple and $\mathbf{f}(\textsc{e})$ is not a null vector (because $f_0(\textsc{e})=1$),
it is clear that the corresponding eigenvectors in Proposition \ref{evec:prp} provide an eigenbasis diagonalizing $\mathbf{H}$:
\begin{subequations}
\begin{equation}\label{diag:a}
\mathbf{F}^{-1}\mathbf{H}\mathbf{F} =\mathbf{E}
\end{equation}
with
\begin{equation}\label{diag:b}
 \mathbf{F}=\bigl[\mathbf{f}(\textsc{e}_0), \mathbf{f}(\textsc{e}_1),\ldots ,\mathbf{f}(\textsc{e}_\textsc{m})\bigr]  \quad
 \text{and}\quad
 \mathbf{E}=\text{diag} (\textsc{e}_0,\textsc{e}_1,\ldots,\textsc{e}_{\textsc{m}}) .
\end{equation}
\end{subequations}
Moreover, upon pulling out the normalization constants from the rows of $\mathbf{F}$ and bringing the resulting
$(\textsc{m}+1)\times (\textsc{m}+1)$ matrix of monic elliptic Racah polynomials to
Vandermonde form via unitriangular row operations, it is readily seen that
\begin{equation}\label{detF}
\det (\mathbf{F} ) = (-1)^{\frac{1}{2}\textsc{m}(\textsc{m}+1)}   
    \prod_{1\leq l\leq \textsc{m}}\tilde{a}_l^{-l}\prod_{0\leq j<k\leq \textsc{m}}(\textsc{e}_j-\textsc{e}_k) .
\end{equation}

\subsection{Orthogonality relation}

The three-term recurrence relation in Proposition \ref{rec:prp} guarantees that the elliptic Racah polynomials obey
the Christoffel-Darboux formulas from the theory of orthogonal polynomials (cf. \cite[Chapter 3.2]{sze:orthogonal}).

\begin{lemma}[Christoffel-Darboux Formulas]\label{CD:lem}
For any $0\leq n\leq\textsc{m}$,  the polynomials
$p_0,\ldots, p_{n+1}$ enjoy the following Christoffel-Darboux identities:
\begin{subequations}
\begin{equation}\label{CD}
\sum_{k=0}^n\dfrac{p_k(\textsc{x})p_k(\textsc{y})}{\prod_{j=1}^k a_j \tilde{a}_{\textsc{m}+1-j}}=\dfrac{p_{n+1}(\textsc{x})p_n(\textsc{y})-p_n(\textsc{x})p_{n+1}(\textsc{y})}{(\textsc{x}-\textsc{y})\prod_{j=1}^n a_j \tilde{a}_{\textsc{m}+1-j}} ,
\end{equation}
and
\begin{equation}\label{CDc}
\sum_{k=0}^n\dfrac{p_k^2 (\textsc{x})}{\prod_{j=1}^k a_j \tilde{a}_{\textsc{m}+1-j}}=\dfrac{p^\prime_{n+1}(\textsc{x})p_n(\textsc{x})-p^\prime_n(\textsc{x})p_{n+1}(\textsc{x})}{\prod_{j=1}^n a_j \tilde{a}_{\textsc{m}+1-j}}.
\end{equation}
\end{subequations}
\end{lemma}

\begin{proof}
From the three-term recurrence \eqref{rec} it follows that
\begin{align*}
&p_{n+1}(\textsc{x})p_n(\textsc{y})-p_n(\textsc{x})p_{n+1}(\textsc{y})=\\
&(\textsc{x}-\textsc{y})p_n(\textsc{x})p_n(\textsc{y})+
a_n\tilde{a}_{\textsc{m}+1-n}\bigl(p_n(\textsc{x})p_{n-1}(\textsc{y})-p_{n-1}(\textsc{x})p_n(\textsc{y})\bigr) .
\end{align*}
Downward iteration entails that
\begin{equation*}
p_{n+1}(\textsc{x})p_n(\textsc{y})-p_n(\textsc{x})p_{n+1}(\textsc{y})=(\textsc{x}-\textsc{y})\sum_{k=0}^n\Bigl(p_k(\textsc{x})p_k(\textsc{y})\prod_{j=k+1}^n a_j\tilde{a}_{\textsc{m}+1-j}\Bigr).
\end{equation*}
Upon dividing both sides by $(\textsc{x}-\textsc{y})\prod_{j=1}^n a_j\tilde{a}_{\textsc{m}+1-j}$ the Christoffel-Darboux formula in
Eq. \eqref{CD} is immediate, while Eq. \eqref{CDc} follows subsequently via the confluent limit $\textsc{y}\to \textsc{x}$.
\end{proof}

Let
\begin{equation}
\mathbf{\tilde{H}}=\pi_2 \bigl( \mathbf{H}\bigr),\quad
     \mathbf{\tilde{f}}(\textsc{e})= \pi_2\bigl( \mathbf{f}(\textsc{e}) \bigr) ,\quad \tilde{p}_k(\textsc{e})=\pi_2 \bigl(  p_k(\textsc{e})  \bigr),
     \quad\text{and}\ \tilde{\textsc{e}}_j=\pi_2 (\textsc{e}_j).
\end{equation}
From Eq. \eqref{qcs} one learns that the matrices $\mathbf{\tilde{H}}$ and $\mathbf{ H}$ are related by conjugation with the
$(\textsc{m}+1)\times (\textsc{m}+1)$ palindromic-involution matrix:
\begin{equation}\label{J}
   \mathbf{\tilde{H}}= \mathbf{J} \mathbf{H} \mathbf{J}\quad\text{with}\quad  \mathbf{J}=\begin{bmatrix}
0&0&\dots&0&1\\
0&0&\iddots&1&0\\
\vdots&\iddots&\iddots&\iddots&\vdots\\
0&1&\iddots&0&0\\
1&0&\dots&0&0
\end{bmatrix} .
\end{equation}

\begin{lemma}[Palindromic Quasi Symmetry]\label{pqs:lem}
\begin{subequations}
\emph{(i)} For any $0\leq j\leq \textsc{m}$, one has that
\begin{equation}\label{quasi-palindromy}
\tilde{\textsc{e}}_j=\textsc{e}_j\quad \text{and}\quad 
 \mathbf{\tilde{f}}(\textsc{e}_j) =\epsilon_j \mathbf{J} \mathbf{f}(\textsc{e}_j)\quad
 \text{with}\ \epsilon_j= \frac{\tilde{p}_{\textsc{m}}(\textsc{e}_j)}{a_1\cdots a_{\textsc{m}}}.
\end{equation}

\emph{(ii)} Let $\tilde{\epsilon}_j=\pi_2(\epsilon_j)= p_{\textsc{m}}(\textsc{e}_j)/(\tilde{a}_1\cdots \tilde{a}_{\textsc{m}})$.
Then
\begin{equation}\label{epsj}
 \epsilon_j\tilde{\epsilon}_j=1\quad\text{and}\quad   \mathrm{Sign}\, (\epsilon_j )=\mathrm{Sign}\, (\tilde{\epsilon}_j )=(-1)^j.
\end{equation}
\end{subequations}
\end{lemma}

\begin{proof}
Because  $\tilde{\mathbf{H}}$ and  $\mathbf{H}$   are related by a similarity transformation,
it is clear that $\tilde{\textsc{e}}_j =\textsc{e}_j$ and
\begin{equation*}
 \mathbf{H} \mathbf{J} \mathbf{\tilde{f}}(\textsc{e}_j)= \mathbf{J}  \mathbf{\tilde{H}}  \mathbf{\tilde{f}}(\textsc{e}_j)
  =  \textsc{e}_j \mathbf{J} \mathbf{\tilde{f}}(\textsc{e}_j).
\end{equation*}
Since the eigenvalue $\textsc{e}_j$ is simple, this
implies that $\mathbf{J}\mathbf{\tilde{f}}(\textsc{e}_j)=\epsilon_j \mathbf{f}(\textsc{e}_j)$ for some  constant $\epsilon_j=\epsilon_j(u_r,v_r)$ in
$ \mathbb{R}$.   
Upon comparing the last components on both sides of  the second equality in Eq. \eqref{quasi-palindromy}, we see that
$\epsilon_j= \epsilon_j f_0(\textsc{e}_j)=\tilde{f}_{\textsc{m}}(\textsc{e}_j)= \tilde{p}_{\textsc{m}}(\textsc{e}_j)/(a_1\cdots a_{\textsc{m}})$, which proves \emph{(i)}. 

Twice iterated application of Eq. \eqref{quasi-palindromy} shows that
$\mathbf{\tilde{f}}(\textsc{e}_j)=\epsilon_j\mathbf{J} \mathbf{f}(\textsc{e}_j)=\epsilon_j \tilde{\epsilon}_j \mathbf{\tilde{f}}(\textsc{e}_j)$ (since $\mathbf{J}^2=\mathbf{I}_{\textsc{m}+1}$), so
 $\epsilon_j \tilde{\epsilon}_j =1$ (as
$\mathbf{\tilde{f}}(\textsc{e}_j)$ is not a null vector). To compute the sign of $\epsilon_j$, 
we
evaluate the Christoffel-Darboux identity in Eq. \eqref{CDc} for $n=\textsc{m}$ at $\textsc{x}=\textsc{e}_j$ by  means of the factorization $p_{\textsc{m}+1}(\textsc{x})= (\textsc{x}-\textsc{e}_0)(\textsc{x}-\textsc{e}_1)\cdots (\textsc{x}-\textsc{e}_{\textsc{m}})$; this reveals
that
$
p_{\textsc{m}}(\textsc{e}_j) \prod_{\substack{0\leq l\leq \textsc{m}\\ l\neq j}}
( \textsc{e}_j-\textsc{e}_l )=p_{\textsc{m}}(\textsc{e}_j)p^\prime_{\textsc{m}+1}(\textsc{e}_j) >0 ,
$
so
\begin{equation*}
\text{Sign} (\epsilon_j )=\text{Sign} (\tilde{\epsilon}_j ) =\text{Sign}  \bigl( p_{\textsc{m}}(\textsc{e}_j) \bigr)
= \text{Sign} \Bigl( \prod_{\substack{0\leq l\leq \textsc{m}\\ l\neq j}} ( \textsc{e}_j-\textsc{e}_l )\Bigr) =(-1)^{j} 
\end{equation*}
(upon recalling the ordering of the eigenvalues from Eq. \eqref{eigenvalues}). This completes the proof of \emph{(ii)}.
\end{proof}

Notice that it follows from the relation $\epsilon_j\tilde{\epsilon}_j=1$ in Lemma \ref{pqs:lem} that
\begin{equation}\label{pM}
p_{\textsc{m}} (\textsc{e}_j) \tilde{p}_{\textsc{m}}(\textsc{e}_j) = \prod_{1\leq k\leq\textsc{m}} a_k\tilde{a}_k .
\end{equation}
Moreover, if
\begin{subequations}
\begin{equation}\label{cs-cond}
(u_1,v_1)=(u_2,v_2)\quad \text{and}\quad (u_3,v_3)=(u_4,v_4),
\end{equation}
then our matrix becomes centrosymmetric: $\mathbf{J}\mathbf{H}\mathbf{J}=\mathbf{\tilde{H}}=\mathbf{H}$. We then have that $\tilde{\epsilon}_j=\epsilon_j$ with $ \epsilon_j^2=1$, so 
\begin{equation}
\epsilon_j=(-1)^j
\end{equation}
\end{subequations}
on this particular parameter manifold enjoying palindromic symmetry.

The orthogonality relations for the elliptic Racah polynomials are governed by the positive weights
\begin{align}\label{weights}
 \Delta_k&=\prod_{1\leq l \leq k}  \frac{\tilde{a}_{\textsc{m}+1-l}}{a_l} \qquad\qquad
 (\text{for}\ k=0,1\ldots,\textsc{m}) \\
 &= {\textstyle  \frac{[2u_1+2k]_1}{[2u_1]_1} }
 \prod_{\substack{ 1\leq l \leq k\\ 1\leq r\leq 4 }}  
 {\textstyle \frac{[u_2-u_{\pi_2(r)}+\textsc{m}+1-l]_r [u_2-v_{\pi_2(r)}+\textsc{m}+\frac{1}{2}-l]_r}{[u_1-u_r+l]_r [u_1-v_r-\frac{1}{2}+l]_r}} \nonumber
\end{align}
 (where the expression was simplified using the duplication formula $[2z]_1=2\prod_{1\leq r\leq 4} [z]_r$ for the scaled theta functions).

\begin{proposition}[Orthogonality Relation]\label{or:prp}
The normalized elliptic Racah polynomials $f_k(\textsc{e})$ \eqref{evec} satisfy
the orthogonality relation
\begin{subequations}
\begin{equation}\label{or:a}
\sum_{k=0}^\textsc{m} f_k(\textsc{e}_i) f_k(\textsc{e}_j) \Delta_k
=\begin{cases}
N_j&\text{if}\ i=j, \\
0&\text{if}\ i\neq j,
\end{cases}
\end{equation}
for $0\leq i,j\leq\textsc{m}$, with
\begin{align}\label{or:b}
N_j=&\frac{1}{|\epsilon_j |\, a_1\cdots a_{\textsc{m}}}\prod_{\substack{0\leq l\leq \textsc{m} \\ l\neq j}}| \textsc{e}_j-\textsc{e}_l| \\
=& \frac{1}{|\epsilon_j|} \prod_{\substack{ 1\leq k\leq \textsc{m}\\ 1\leq r\leq 4}}{\textstyle \frac{[u_1+k]_r}{[u_1-u_r+k]_r} \frac{[u_1-\frac{1}{2}+k]_r}{[u_1-v_r-\frac{1}{2}+k ]_r}}
\prod_{\substack{0\leq l\leq \textsc{m} \\ l\neq j}}| \textsc{e}_j-\textsc{e}_l| . \nonumber 
\end{align}
\end{subequations}
\end{proposition}

\begin{proof}
The asserted orthogonality follows by combining the Christoffel-Darboux formula of
Lemma \ref{CD:lem} for $n=\textsc{m}+1$ with Eq. \eqref{charpol}:
\begin{equation*}
 \sum_{k=0}^{\textsc{m}}  f_k(\textsc{e}_i) f_k(\textsc{e}_j)  \Delta_k 
     = \sum_{k=0}^{\textsc{m}} \dfrac{p_k(\textsc{e}_i) p_k(\textsc{e}_j)}{\prod_{1\leq j \leq k} a_j \tilde{a}_{\textsc{m}+1-j}}  =
     \begin{cases}
    \dfrac{p^\prime_{\textsc{m}+1}(\textsc{e}_j) p_{\textsc{m}}(\textsc{e}_j)}{\prod_{1\leq j \leq \textsc{m}}a_j \tilde{a}_{\textsc{m}+1-j}}&\text{if}\ i=j,\\
    \qquad 0&\text{if}\ i\neq j.
     \end{cases}
\end{equation*}
Indeed, since $p^\prime_{\textsc{m}+1}(\textsc{e}_j)=\prod_{\substack{0\leq l\leq \textsc{m}\\ l\neq j}} ( \textsc{e}_j-\textsc{e}_l )$
and
$ p_{\textsc{m}}(\textsc{e}_j)=  \tilde{a}_1\cdots \tilde{a}_{\textsc{m}} /\epsilon_j $ (by Eqs. \eqref{quasi-palindromy},
\eqref{epsj}), 
the expression for the quadratic norm readily simplifies to the formula stated in the lemma.
\end{proof}

The orthogonality relation in Proposition \ref{or:prp} supplies the following expressions for
the inverse and the determinant of the elliptic Racah matrix
$\mathbf{F}$ from Eqs. \eqref{diag:a}, \eqref{diag:b}.

\begin{corollary}[Inverse Elliptic Racah Matrix]\label{Finv:cor}
The inverse and the determinant of the elliptic Racah matrix $\mathbf{F}$ \eqref{diag:b} are given by
\begin{subequations}
\begin{equation}\label{Finv}
\mathbf{F}^{-1}= \mathbf{N}^{-1} \mathbf{F}^T \boldsymbol{\Delta}
\end{equation}
and
\begin{equation}\label{Fdet}
    \det (\mathbf{F})=(-1)^{\frac{1}{2}\textsc{m}(\textsc{m}+1)} \prod_{0\leq l \leq \textsc{m}}
    {\textstyle \left(\frac{N_l}{\Delta_l}\right)^{1/2} },
\end{equation}
\end{subequations}
where $\mathbf{N}=\mathrm{diag}\, (N_0,N_1,\ldots,N_\textsc{m})$,
$\boldsymbol{\Delta}=\mathrm{diag}\, (\Delta_0,\Delta_1,\ldots,\Delta_{\textsc{m}})$
and 
\begin{equation*}
  \mathbf{F}^T=
 \begin{bmatrix}
f_0(\textsc{e}_0) & f_1(\textsc{e}_0)&  \cdots & f_{\textsc{m}}(\textsc{e}_0) \\
f_0(\textsc{e}_1) & f_1(\textsc{e}_1)&  \cdots & f_{\textsc{m}}(\textsc{e}_1) \\
\vdots &  \vdots &  \vdots  & \vdots \\
f_0(\textsc{e}_{\textsc{m}-1}) & f_1(\textsc{e}_{\textsc{m}-1})&  \cdots & f_{\textsc{m}}(\textsc{e}_{\textsc{m}-1})  \\
f_0(\textsc{e}_{\textsc{m}}) & f_1(\textsc{e}_{\textsc{m}})&  \cdots & f_{\textsc{m}}(\textsc{e}_{\textsc{m}})
\end{bmatrix} .
\end{equation*}
\end{corollary}

If one compares Eq. \eqref{Fdet} with the evaluation of the determinant in  Eq. \eqref{detF}, then it follows that
\begin{subequations}
\begin{equation}
\epsilon_0\epsilon_1\cdots\epsilon_{\textsc{m}}= (-1)^{\frac{1}{2}\textsc{m}(\textsc{m}+1)} 
  \prod_{1\leq l\leq \textsc{m}}  \left( \frac{\tilde{a}_l}{a_l} \right)^{l} ,
\end{equation}
or equivalently (cf. Eq. \eqref{pM})
\begin{equation}
p_{\textsc{m}}(\textsc{e}_0)p_{\textsc{m}}(\textsc{e}_1)\cdots p_{\textsc{m}}(\textsc{e}_\textsc{m})=
(-1)^{\frac{1}{2}\textsc{m}(\textsc{m}+1)} 
  \prod_{1\leq l\leq \textsc{m}}   (a_l \tilde{a}_{\textsc{m}+1-l})^l .
\end{equation}
\end{subequations}

\subsection{Finite discrete Heun function}

In order to describe the complete solution of the
finite discrete Heun equation \eqref{dh:a}, \eqref{dh:b} in terms of elliptic Racah polynomials,
the following theorem summarizes the main findings of this section.

\begin{theorem}[Finite Discrete Heun Function]\label{dhsol:thm}

\begin{subequations}
\emph{(i)} The finite discrete Heun equation \eqref{dh:a}, \eqref{dh:b} only possesses nontrivial solutions for $\textsc{e}\in \{ \textsc{e}_0,\ldots ,\textsc{e}_{\textsc{m}}\}$, where $\textsc{e}_0>\textsc{e}_1>\cdots>\textsc{e}_{\textsc{m}}$ 
 denote the roots of the top-degree elliptic Racah polynomial $p_{\textsc{m}+1}(\textsc{e})$ \eqref{e-racah}.

\emph{(ii)}  The solutions of the finite discrete Heun equation from part \emph{(i)} are given by
\begin{equation}\label{fh:a}
h_k (\textsc{e}_j)=h_0(\textsc{e}_j) p_k(\textsc{e}_j) \prod_{0\leq j<k} \tilde{a}_{\textsc{m}-j}^{-1}  \qquad (k=0,\ldots ,\textsc{m}) 
\end{equation}
with
\begin{equation}\label{fh:b}
h_0(\textsc{e}_j) =|\epsilon_j|^{1/2} =
\left|\frac{\tilde{a}_1\cdots\tilde{a}_\textsc{m}}{p_\textsc{m}(\textsc{e}_j)}\right|^{1/2}
\end{equation}
(cf. Proposition \ref{evec:prp}).

\emph{(iii)} The finite discrete Heun function \eqref{fh:a}, \eqref{fh:b} satisfies the 
 palindromic quasi-symmetry
\begin{equation}
 \tilde{h}_k(\tilde{\textsc{e}}_j)= (-1)^{j} h_{\textsc{m}-k} (\textsc{e}_j) \quad\text{with}\quad
  \tilde{\textsc{e}}_j=\textsc{e}_j\quad  (j,k=0,\ldots ,\textsc{m}),
\end{equation}
where  $\tilde{h}_k(\tilde{\textsc{e}}_j)$ denotes  the finite discrete Heun function
with permuted coupling parameters:
$(u_1,v_1) \leftrightarrow (u_2,v_2)$ and $(u_3,v_3) \leftrightarrow (u_4,v_4)$ (cf. Lemma \ref{pqs:lem}).

\emph{(iv)} The finite discrete Heun functions satisfy the orthogonality relation
\begin{align}
  & \sum_{k=0}^{\textsc{m}}  h_k(\textsc{e}_i) h_k(\textsc{e}_j)  
 \prod_{1\leq j \leq k} \tilde{a}_{\textsc{m}+1-j}   \prod_{k+1\leq j\leq\textsc{m}} {\textstyle a_j}   \\
   &=
   \begin{cases}
   \prod_{\substack{0\leq l\leq \textsc{m}\\ l\neq j}} | \textsc{e}_j-\textsc{e}_l |  &\text{if}\ i=j,\\
   \ 0 &\text{if}\ i\neq j
   \end{cases} \nonumber
  \end{align}
  ($0\leq i,j\leq\textsc{m}$), and the dual orthogonality relation
  \begin{equation}
   \sum_{j=0}^{\textsc{m}}  
   \frac{ h_l(\textsc{e}_j) h_k(\textsc{e}_j) } {\prod_{\substack{0\leq i\leq \textsc{m}\\ i\neq j}} | \textsc{e}_j-\textsc{e}_i |}
   =
   \begin{cases}
    \left(\prod_{1\leq j \leq k} \tilde{a}_{\textsc{m}+1-j}   \prod_{k+1\leq j\leq\textsc{m}} {\textstyle a_j}\right)^{-1}
  &\text{if}\ l=k,\\
   \ 0 &\text{if}\ l\neq k
   \end{cases} 
  \end{equation}
 ($0\leq l,k\leq\textsc{m}$), which encode respectively the row and column orthogonality of the elliptic Racah matrix
 $\mathbf{F}$ \eqref{diag:b} (cf. Proposition \ref{or:prp}).
\end{subequations}
\end{theorem}

\section{Degenerations}\label{sec4}

\subsection{Finite discrete Lam\'e equation}
If all parameters $v_r$ tend to zero, then the difference Heun operator $\mathrm{H}$ \eqref{dH:eq}--\eqref{cr} reduces to a
second-order difference operator stemming from the Sklyanin algebra \cite{ros:sklyanin,skl:some,spi:continuous,rai-rui:difference}:
\begin{equation}
(\mathrm{H}f)(z)= f(z+1) \prod_{1\leq r\leq 4}{\textstyle \frac{[z+u_r]_r}{[z]_r} } +
f(z-1) \prod_{1\leq r\leq 4}{\textstyle \frac{[z+u_r]_r}{[z]_r} }  .
\end{equation}
The corresponding parameter degeneration of Theorem \ref{dhsol:thm} solves a finite discrete Heun equation of the form
\begin{subequations}
\begin{equation}\label{ds:a}
\tilde{a}_{\textsc{m}-k} f_{k+1}+ a_k f_{k-1} =  \textsc{e} f_k\quad\text{for}\ k=0,\ldots ,\textsc{m},
\end{equation}
with
\begin{equation}\label{ds:b}
a_k = \prod_{1\leq r\leq 4}{\textstyle \frac{[u_1-u_r+k]_r}{[u_1+k]_r}}
\quad\text{and}\quad
\tilde{a}_k = \prod_{1\leq r\leq 4}{\textstyle \frac{[u_2-u_{\pi_2(r)}+k]_r}{[u_2+k]_r}} .
\end{equation}
\end{subequations}
Via the duplication formula $[2z]_1=2\prod_{1\leq r\leq 4}[z]_r$ it is seen that
if all parameters $u_r$ are equal (to $u>0$ say), then the eigenvalue problem in Eqs. \eqref{ds:a}, \eqref{ds:b}
reduces to a finite discrete Lam\'e equation of the form studied in \cite{die-gor:elliptic}:
\begin{equation}
 {\textstyle \frac{[\textsc{m}-k]_1}{[u+\textsc{m}-k]_1}} f_{k+1}+  {\textstyle \frac{[k]_1}{[u +k]_1}} f_{k-1} =  \textsc{e} f_k\quad\text{for}\ k=0,\ldots ,\textsc{m},
\end{equation}
with $\alpha =\frac{2\pi}{2u+\textsc{m}}$ (so $2u+\textsc{m}=\frac{2\pi}{\alpha}$).

\subsection{Trigonometric limit: \emph{q}-Racah polynomials}
In \cite{die-gor:elliptic} it was shown that the solutions of the finite discrete Lam\'e equation \eqref{ds:a}, \eqref{ds:b} can be expressed in terms of Rogers' $q$-ultraspherical polynomials in the trigonometric limit $p\to 0$.
Here we finish by checking that  the elliptic Racah polynomials degenerate in turn to the  $q$-Racah polynomials of Askey and Wilson in the limit $p\to 0$. 
To this end let us first observe that in the trigonometric limit the scaled theta functions degenerate as follows:
\begin{equation*}
\lim_{p\to 0}[z]_1=\tfrac{2}{\alpha}\sin(\tfrac{\alpha}{2}z),\quad
\lim_{p\to 0}[z]_2=\cos(\tfrac{\alpha}{2}z),\quad
\lim_{p\to 0}[z]_3=1,\quad
\lim_{p\to 0}[z]_4=1 .
\end{equation*}
The corresponding coefficients of the difference Heun equation thus become
\begin{subequations}
\begin{equation}\label{A-trig}
\textsc{a}_{\mathrm{t}} (z)=\lim_{p\to 0} \textsc{a} (z)={\textstyle
\frac{\sin(\frac{\alpha}{2}(z+u_1))}{\sin(\frac{\alpha}{2}z)}
\frac{\cos(\frac{\alpha}{2}(z+u_2))}{\cos(\frac{\alpha}{2}z)}
\frac{\sin (\frac{\alpha}{2}(z+ \frac{1}{2}+v_1)}{\sin(\frac{\alpha}{2}(z+\frac{1}{2}))}
\frac{\cos(\frac{\alpha}{2}(z+\frac{1}{2}+v_2))}{\cos(\frac{\alpha}{2}(z+\frac{1}{2}))} }
\end{equation}
and
\begin{align}\label{B-trig}
\textsc{b}_{\mathrm{t}} (z)=\lim_{p\to 0} \textsc{b} (z)=&
c_{\mathrm{t},1}
{\textstyle \frac{\sin(\frac{\alpha}{2}(z+\frac{1}{2}+u))}{\sin(\frac{\alpha}{2}(z+\frac{1}{2}))}\frac{\sin(\frac{\alpha}{2}(z-\frac{1}{2}-u))}{\sin(\frac{\alpha}{2}(z-\frac{1}{2}))} }\\&
+c_{\mathrm{t},2}
{\textstyle \frac{\cos(\frac{\alpha}{2}(z+\frac{1}{2}+u))}{\cos(\frac{\alpha}{2}(z+\frac{1}{2}))}\frac{\cos(\frac{\alpha}{2}(z-\frac{1}{2}-u))}{\cos(\frac{\alpha}{2}(z-\frac{1}{2}))}}
+c_{\mathrm{t},3}+c_{\mathrm{t},4} \nonumber
\end{align}
\end{subequations}
with
\begin{equation*}
c_{\mathrm{t},r}=
{\textstyle \frac{2\sin(\frac{\alpha}{2}(u_{\pi_r(1)}-\frac{1}{2}) )\sin(\frac{\alpha}{2}v_{\pi_r(1)})
\cos(\frac{\alpha}{2}(u_{\pi_r(2)}-\frac{1}{2}) ) \cos(\frac{\alpha}{2}v_{\pi_r(2)})}{\sin(\frac{\alpha}{2}u)\sin(\frac{\alpha}{2}(u+1))} }.
\end{equation*}
We now have that
\begin{equation}\label{Ct}
\textsc{a}_{\mathrm{t}} (z)+\textsc{a}_{\mathrm{t}} (-z)+\textsc{b}_{\mathrm{t}} (z)=\textsc{c}_{\mathrm{t}}=
{\textstyle 2\cos \frac{\alpha}{2} (u_1+u_2+v_1+v_2)}+ \sum_{1\leq r\leq 4} c_{\mathrm{t},r} .
\end{equation}
Indeed, as a periodic function of $z$ all poles on the LHS of Eq. \eqref{Ct} are seen to cancel, while for $\text{Im}\, (z)\to \infty$ the expression in question tends to the constant value on the RHS.

The $q$-Racah polynomials \cite{ask-wil:set} are basic hypergeometric orthogonal polynomials of the form \cite[Chapter 4.2]{koe-les-swa:hypergeometric}:
\begin{subequations}
\begin{equation}
R_k(\textsc{x})=R_k(\textsc{x}(x);\mathrm{a},\mathrm{b},\mathrm{c},\mathrm{d} | q )=
{}_4\phi_3 \bigl( \substack{q^{-k},  \mathrm{ab} q^{k+1}, q^{-x}, \mathrm{cd}q^{x+1} \\ \mathrm{a}q,\mathrm{bd}q,\mathrm{c}q } ;q,q \bigr)
\end{equation}
with
\begin{equation}
\textsc{x}=\textsc{x}(x)=\mathrm{cd}q^{x+1}+q^{-x} .
\end{equation}
\end{subequations}
From the basic hypergeometric representation the $k\leftrightarrow j$, $\mathrm{a}\leftrightarrow \mathrm{c}$, $\mathrm{b}\leftrightarrow \mathrm{d}$ duality symmetry \cite{ask-wil:set,leo:orthogonal}
of the $q$-Racah polynomial $R_k(\textsc{x}(j);\mathrm{a},\mathrm{b},\mathrm{c},\mathrm{d} | q )$ is immediate:
\begin{subequations}
\begin{equation}\label{qRdual:a}
R_k(\textsc{x}(j);\mathrm{a},\mathrm{b},\mathrm{c},\mathrm{d} | q )=
R_j(\hat{\textsc{x}}(k);\mathrm{\hat a},\mathrm{\hat b},\mathrm{\hat c},\mathrm{\hat d} | q ),
\end{equation}
with
\begin{equation}\label{qRdual:b}
\hat{\textsc{x}}(x)=\mathrm{\hat{c} \hat{d}}q^{x+1}+q^{-x} \quad\text{and}\quad
(\mathrm{\hat a},\mathrm{\hat b},\mathrm{\hat c},\mathrm{\hat d})=(\mathrm{c},\mathrm{d},\mathrm{a},\mathrm{b})  .
\end{equation}
\end{subequations}
The following proposition recovers the $q$-Racah polynomials as a trigonometric limit of the elliptic Racah polynomials.

\begin{proposition}[$q$-Racah Limit]\label{qR:prp}
For $1\leq j,k\leq \textsc{m}$ and parameters in accordance with Eqs. \eqref{pd:a}, \eqref{pd:b}, one has that
\begin{subequations}
\begin{equation}\label{qR:a}
\lim_{p\to 0}\textsc{e}_j=\textsc{e}_{\mathrm{t},j}= {\textstyle 2\cos \frac{\alpha}{2} (2j+u_1+u_2+v_1+v_2)}+\sum_{1\leq r\leq 4} c_{\mathrm{t},r}
\end{equation}
and
\begin{align}\label{qR:b}
\lim_{p\to 0} f_k(\textsc{e})&=f_{\mathrm{t},k}(\textsc{e})= R_k(\textsc{x}(x);\mathrm{a},\mathrm{b},\mathrm{c},\mathrm{d} | q ) \\
&=
{}_4\phi_3 \left( \substack{q^{-k},   q^{2u_1+k}, q^{-x}, q^{u_1+u_2+v_1+v_2+x} \\ -q^{u_1+u_2},q^{u_1+v_1+1/2},-q^{ u_1+v_2+1/2}} ;q,q \right),
\nonumber
\end{align}
where 
\begin{equation}\label{qR:c}
q=e^{\mathrm{i}\alpha},\quad \textsc{e}=  {\textstyle 2\cos \frac{\alpha}{2} (2x+u_1+u_2+v_1+v_2)}+\sum_{1\leq r\leq 4} c_{\mathrm{t},r} 
\end{equation}
and
\begin{equation}\label{qR:d}
 \mathrm{a}= -q^{u_1+u_2-1},\  \mathrm{b}=-q^{u_1-u_2}  ,\  \mathrm{c}=-q^{u_1+v_2-1/2},\mathrm{d}=-q^{u_2+v_1-1/2} .
\end{equation}
\end{subequations}
\end{proposition}

\begin{proof}
From the recurrence relation for the normalized elliptic Racah polynomials,
it follows that in the trigonometric limit:
\begin{equation}\label{t-rec}
\tilde{a}_{\mathrm{t},\textsc{m}-k}  f_{\mathrm{t},k+1}(\textsc{e})+ a_{\mathrm{t},k} f_{\mathrm{t},k-1}(\textsc{e})+
(\textsc{c}_{\mathrm{t}}-\tilde{a}_{\mathrm{t},\textsc{m}-k}-a_{\mathrm{t},k}) f_{\mathrm{t},k}(\textsc{e})= \textsc{e} f_{\mathrm{t},k}(\textsc{e})
\end{equation}
for $0\leq k <\textsc{m}$, with 
\begin{align*}
a_{\mathrm{t},k}&=\textsc{a}_{\mathrm{t}}(-u_1-k)= {\textstyle \frac{\sin\frac{\alpha}{2}(k)}{\sin\frac{\alpha}{2} (u_1+k) }
\frac{\sin \frac{\alpha}{2}(u_1-v_1-\frac{1}{2}+k )}{\sin\frac{\alpha}{2}(u_1-\frac{1}{2}+k)}}
{\textstyle \frac{\cos\frac{\alpha}{2}(u_1-u_2+k)}{\cos\frac{\alpha}{2}(u_1+k) }\frac{\cos\frac{\alpha}{2}(u_1-v_2-\frac{1}{2}+k) }{\cos\frac{\alpha}{2}(u_1-\frac{1}{2}+k)}} , \\
\tilde{a}_{\mathrm{t},k}&=\textsc{a}_{\mathrm{t}}(u_1+\textsc{m}-k)= 
 {\textstyle \frac{\sin\frac{\alpha}{2}(k)}{\sin\frac{\alpha}{2} (u_2+k) }
\frac{\sin \frac{\alpha}{2}(u_2-v_2-\frac{1}{2}+k )}{\sin\frac{\alpha}{2}(u_2-\frac{1}{2}+k)}}
{\textstyle \frac{\cos\frac{\alpha}{2}(u_2-u_1+k)}{\cos\frac{\alpha}{2}(u_2+k) }\frac{\cos\frac{\alpha}{2}(u_2-v_1-\frac{1}{2}+k) }{\cos\frac{\alpha}{2}(u_2-\frac{1}{2}+k)}} ,
\end{align*}
and  the coefficient $\textsc{b}_{\mathrm{t},k} =\textsc{b}_{\mathrm{t}}(u_1+ k)$ has been rewritten with the aid of the identity in Eq. \eqref{Ct}.
Upon comparing with the three-term recurrence relation for the $q$-Racah polynomials \cite[Eq. (14.2.3)]{koe-les-swa:hypergeometric}:
\begin{equation}\label{qR-rec}
A_k R_{k+1}(\textsc{x})+ C_kR_{k-1}(\textsc{x})+(\mathrm{c}\mathrm{d}q+1-A_k-C_k)R_k(\textsc{x})=\textsc{x}R_k(\textsc{x}),
\end{equation}
with
\begin{equation*}
{\textstyle A_k=\frac{(1-\mathrm{a}q^{k+1})(1-\mathrm{ab}q^{m+1})(1-\mathrm{bd}q^{k+1})(1-\mathrm{c}q^{k+1})}{(1-\mathrm{ab}q^{2k+1})(1-\mathrm{ab}q^{2k+2})}}
\end{equation*}
and
\begin{equation*}
{\textstyle C_k=\frac{q(1-q^k)(1-\mathrm{b}q^k)(\mathrm{c}-\mathrm{ab}q^k)(\mathrm{d}-\mathrm{a}q^k)}{(1-\mathrm{ab}q^{2k})(1-\mathrm{ab}q^{2k+1})} },
\end{equation*}
one observes that $\tilde{a}_{\mathrm{t},\textsc{m}-k} =(q\mathrm{cd})^{-1/2}A_k$, $ a_{\mathrm{t},k} =(q\mathrm{cd})^{-1/2}C_k$, and $\textsc{e}-\textsc{c}_{\mathrm{t}}=
(q\mathrm{cd})^{-1/2}(\textsc{x}(x)-\mathrm{cd}q-1)$ provided the variables and
parameters are identified in accordance with Eqs. \eqref{qR:c}, \eqref{qR:d}.  The upshot is that both recurrences coincide while $f_{\mathrm{t},0}(\textsc{e})=R_0(\textsc{x})=1$,
so Eq. \eqref{qR:b} follows.

To infer the limit in Eq. \eqref{qR:a} it suffices to check that the eigenvalues of the trigonometric degeneration of finite discrete Heun operator are indeed given by the asserted formulas on the RHS, or equivalently, that
the recurrence relation  for  $f_{\mathrm{t},k}(\textsc{e})$ in Eq. \eqref{t-rec} remains valid for $k=\textsc{m}$
if  $\textsc{e}$ belongs to $\{ \textsc{e}_{\mathrm{t},0},\ldots ,\textsc{e}_{\mathrm{t},\textsc{m}}\}$
(with $\textsc{e}_{\mathrm{t},j}$ given by the RHS of Eq.  \eqref{qR:a}). 
To this end we deduce from the $q$-difference equation \cite[Eq. (14.2.6)]{koe-les-swa:hypergeometric} satisfied by the dual $q$-Racah polynomial
$\hat{R}_j(\hat{\textsc{x}}(k))=R_j(\hat{\textsc{x}}(k);\mathrm{\hat a},\mathrm{\hat b},\mathrm{\hat c},\mathrm{\hat c} | q )$ \eqref{qRdual:a}, \eqref{qRdual:b} that
 for any $0\leq j,k\leq\textsc{m}$:
\begin{equation*}
A_k \hat{R}_{j}(\hat{\textsc{x}}(k+1))+ C_k\hat{R}_{j}(\hat{\textsc{x}}(k-1))+(\mathrm{c}\mathrm{d}q+1-A_k-C_k)\hat{R}_j(\hat{\textsc{x}}(k))=\textsc{x}(j) \hat{R}_j(\hat{\textsc{x}}(k)),
\end{equation*}
where $A_{\textsc{m}}=0$ (because $1-\mathrm{a}q^{\textsc{m}+1}=1+q^{u_1+u_2+\textsc{m}}=0$) and $C_0=0$.
In view of the duality symmetry in Eqs. \eqref{qRdual:a}, \eqref{qRdual:b}, this confirms that on shell at
$\textsc{x}=\textsc{x}(j)$ ($0\leq j\leq \textsc{m}$) the recurrence relation in Eq. \eqref{qR-rec} holds for $0\leq k\leq \textsc{m}$. After rewriting the formula in terms of $f_{\mathrm{t},k}(\textsc{e})$ with the aid of Eqs. \eqref{qR:b}--\eqref{qR:d}, we see that the same is therefore true for
the recurrence relation  in Eq. \eqref{t-rec} at  $\textsc{e}=\textsc{e}_{\mathrm{t},j}$,
$j=0,1,\ldots,\textsc{m}$.
\end{proof}

Proposition \ref{qR:prp} reveals that on shell
the trigonometric degeneration of the normalized elliptic Racah polynomial $f_k(\textsc{e}_j)$ is given by the following $q$-Racah polynomial
\begin{align}
f_{\mathrm{t},k}(\textsc{e}_{\mathrm{t},j})&= 
{\textstyle 
R_k\left(\textsc{x}(j);-q^{u_1+u_2-1},-q^{u_1-u_2},-q^{u_1+v_2-1/2},-q^{u_2+v_1-1/2} | q \right) } \\
&={}_4\phi_3 \left( \substack{q^{-k},   q^{2u_1+k}, q^{-j}, q^{u_1+u_2+v_1+v_2+j} \\ -q^{u_1+u_2},q^{u_1+v_1+1/2},-q^{ u_1+v_2+1/2}} ;q,q \right), \nonumber
\end{align}
with $\textsc{x}(j)=q^{u_1+u_2+v_1+v_2+j}+q^{-j}$. The corresponding degeneration of the orthogonality relation from Proposition \ref{or:prp} becomes
 \begin{subequations}
\begin{equation}\label{ort:a}
\sum_{k=0}^\textsc{m} f_{\mathrm{t},k}(\textsc{e}_{\mathrm{t},i}) f_{\mathrm{t},k}(\textsc{e}_{\mathrm{t},j}) \Delta_{\mathrm{t},k}
=\begin{cases}
N_{\mathrm{t},j}&\text{if}\ i=j, \\
0&\text{if}\ i\neq j,
\end{cases}
\end{equation}
for $0\leq i,j\leq\textsc{m}$, with
\begin{align}\label{ort:b}
\Delta_{\mathrm{t},k}&=
 {\textstyle  \frac{\sin \alpha (u_1+k) }{\sin (\alpha u_1)} } \times\\
 &
 \prod_{ 1\leq l \leq k}  
 {\textstyle \frac{\sin\frac{\alpha}{2}(\textsc{m}+1-l) \sin\frac{\alpha}{2}(u_2-v_2+\textsc{m}+\frac{1}{2}-l)}{\sin (\frac{\alpha}{2}l) \sin\frac{\alpha}{2}(u_1-v_1-\frac{1}{2}+l)}}
 {\textstyle \frac{\cos\frac{\alpha}{2}(u_2-u_1+\textsc{m}+1-l) \cos\frac{\alpha}{2}(u_2-v_1+\textsc{m}+\frac{1}{2}-l)}{\cos\frac{\alpha}{2}(u_1-u_2+l) \cos\frac{\alpha}{2}(u_1-v_2-\frac{1}{2}+l)}} \nonumber
\end{align}
and
\begin{align}\label{ort:c}
N_{\mathrm{t},j}=&\frac{1}{\epsilon_{\mathrm{t},j} \, a_{\mathrm{t},1}\cdots a_{\mathrm{t},\textsc{m}}}
\prod_{\substack{0\leq l\leq \textsc{m} \\ l\neq j}} ( \textsc{e}_{\mathrm{t},j}-\textsc{e}_{\mathrm{t},l} ) \\
=& \frac{1}{\epsilon_{\mathrm{t},j}} 
\prod_{ 1\leq k\leq \textsc{m}}{\textstyle 
\frac{\sin\frac{\alpha}{2} (u_1+k)}{\sin\frac{\alpha}{2} k } \frac{\sin\frac{\alpha}{2}(u_1-\frac{1}{2}+k)}{\sin\frac{\alpha}{2}(u_1-v_1-\frac{1}{2}+k )}}
{\textstyle \frac{\cos\frac{\alpha}{2}(u_1+k)}{\cos\frac{\alpha}{2}(u_1-u_2+k)} 
\frac{\cos\frac{\alpha}{2}(u_1-\frac{1}{2}+k)}{\cos\frac{\alpha}{2}(u_1-v_2-\frac{1}{2}+k )}} \nonumber \\
&\times
\prod_{\substack{0\leq l\leq \textsc{m} \\ l\neq j}} \bigl( {\textstyle 2\cos \frac{\alpha}{2} (2j+u_1+u_2+v_1+v_2)}-{\textstyle 2\cos \frac{\alpha}{2} (2l+u_1+u_2+v_1+v_2)} \bigr) , \nonumber 
\end{align}
where
\begin{equation*}
 \epsilon_{\mathrm{t},j}^{-1}=\frac{p_{\mathrm{t},\textsc{m}}(\textsc{e}_{\mathrm{t},j})}{\tilde{a}_{\mathrm{t},1}\cdots\tilde{a}_{\mathrm{t},\textsc{m}}}=   
 f_{\mathrm{t},\textsc{m}}(\textsc{e}_{\mathrm{t},j}) = 
{}_4\phi_3 \left( \substack{q^{-\textsc{m}},   q^{2u_1+\textsc{m}}, q^{-j}, q^{u_1+u_2+v_1+v_2+j} \\ -q^{u_1+u_2},q^{u_1+v_1+1/2},-q^{ u_1+v_2+1/2}} ;q,q \right) .
\end{equation*}
By means of the relation
$q^{u_1+u_2+\textsc{m}}=-1$ we reduce the latter ${}_4\phi_3$ series to a ${}_3\phi_2$ series that can be evaluated via 
Jackson's $q$-Pfaff-Saalsch\"utz sum \cite[Eq. (17.7.4)]{olv-loz-boi-cla:nist}:
\begin{align}
\epsilon_{\mathrm{t},j}^{-1}&=   {}_3\phi_2 \left( \substack{  - q^{u_1-u_2}, q^{-j}, q^{u_1+u_2+v_1+v_2+j} \\ q^{u_1+v_1+1/2},-q^{ u_1+v_2+1/2}} ;q,q \right)= 
{\textstyle\frac{(-q^{u_2+v_1+1/2}, q^{-u_2-v_2+1/2-j};q)_j}{(q^{u_1+v_1+1/2},-q^{-u_1-v_2+1/2-j};q)_j} }  \nonumber \\
&=(-1)^j \prod_{0\leq l < j} 
{\textstyle \frac{\sin\frac{\alpha}{2} (u_2+v_2 +1/2+l)}{\sin\frac{\alpha}{2} (u_1+v_1 +1/2+l)}
\frac{\cos\frac{\alpha}{2} (u_2+v_1 +1/2+l)}{\cos\frac{\alpha}{2} (u_1+v_2 +1/2+l)} } \label{ort:d}
\end{align}
\end{subequations}
(where $(z;q)_j=\prod_{0\leq l <j} (1-zq^l)$ and $(z_1,\ldots ,z_s;q)_j=(z_1;q)_j\cdots (z_s;q)_j$).

It is instructive to compare the orthogonality in Eqs. \eqref{ort:a}--\eqref{ort:d} with the (dual) orthogonality relations for the $q$-Racah polynomials subject to the truncation condition $\mathrm{a}q^{\textsc{m}+1}=1$ (cf. \cite[Eq. (14.2.2)]{koe-les-swa:hypergeometric}):
\begin{subequations}
\begin{align}\label{qRo:a}
\sum_{k=0}^{\textsc{m}}
R_k(\textsc{x}(i);\mathrm{a},\mathrm{b},\mathrm{c},\mathrm{d} | q )
R_k(\textsc{x}(j);\mathrm{a},\mathrm{b},\mathrm{c},\mathrm{d} | q )
\, \Delta_k(\mathrm{a},\mathrm{b},\mathrm{c},\mathrm{d};q)& \\ 
=\begin{cases}
N_0/\Delta_j(\mathrm{c},\mathrm{d},\mathrm{a},\mathrm{b};q) &\text{if}\ i=j\\
0&\text{if}\ i\neq j
\end{cases} 
& \nonumber
\end{align}
with
\begin{equation}\label{qRo:b}
   \Delta_k(\mathrm{a},\mathrm{b},\mathrm{c},\mathrm{d};q)= 
\frac{(\mathrm{c} q,\mathrm{bd} q,\mathrm{a} q,\mathrm{ab} q;q)_k}
     {(q,\mathrm{c}^{-1}\mathrm{ab} q,\mathrm{d}^{-1}\mathrm{a} q,\mathrm{b} q;q)_k}
     \frac{(1-\mathrm{ab} q^{2k+1})}{(\mathrm{cd} q)^k(1-\mathrm{ab} q)}
\end{equation}
and
\begin{equation}\label{qRo:c}
    N_0= \sum_{k=0}^{\textsc{m}} \Delta_k(\mathrm{a},\mathrm{b},\mathrm{c},\mathrm{d};q)= 
    \sum_{j=0}^{\textsc{m}} \Delta_j(\mathrm{c},\mathrm{d},\mathrm{a},\mathrm{b};q)=
    \frac{(\mathrm{b}^{-1},\mathrm{cd}q^2)_{\textsc{m}}}{(\mathrm{b}^{-1}\mathrm{c}q,\mathrm{d}q;q)_{\textsc{m}}}.
\end{equation}
\end{subequations}
Indeed, the orthogonality weights in Eq. \eqref{ort:b} and Eq. \eqref{qRo:b} coincide for parameters in accordance
with Eq. \eqref{qR:c}, \eqref{qR:d} (subject to the truncation condition \eqref{pd:b}): 
\begin{equation*}
\Delta_{\mathrm{t},k}=\Delta_k(-q^{u_1+u_2-1}, -q^{u_1-u_2}  ,-q^{u_1+v_2-1/2},-q^{u_2+v_1-1/2} ;q)
\end{equation*}
when $ q=e^{\text{i}\alpha}$ with $\alpha=\frac{\pi}{u_1+u_2+\textsc{m}}$. This implies that the quadratic
norms  $N_{\mathrm{t},j}$ \eqref{ort:c}, \eqref{ort:d} can be rewritten in the form:
\begin{subequations}
\begin{equation}
  N_{\mathrm{t},j}=N_{\mathrm{t},0}/\hat{\Delta}_{\mathrm{t},j }
\end{equation}
with
\begin{align}
\hat{\Delta}_{\mathrm{t},j} &= \Delta_j (-q^{u_1+v_2-1/2},-q^{u_2+v_1-1/2},-q^{u_1+u_2-1}, -q^{u_1-u_2}  ;q) \\
&= 
 {\textstyle  \frac{\sin \frac{\alpha}{2} (u_1+u_2+v_1+v_2+2j) }{\sin \frac{\alpha}{2} (u_1+u_2+v_1+v_2)} } \times \nonumber \\
 &
 \prod_{ 1\leq l \leq j}  
 {\textstyle \frac{\sin\frac{\alpha}{2}(\textsc{m}+1-l) \sin\frac{\alpha}{2}(u_2-v_2+\textsc{m}+\frac{1}{2}-l)}
 {\sin (\frac{\alpha}{2}l) \sin\frac{\alpha}{2}(u_2+v_2-\frac{1}{2}+l)}}
 {\textstyle \frac{\cos\frac{\alpha}{2}(-v_1-v_2+\textsc{m}+1-l) \cos\frac{\alpha}{2}(u_2-v_1+\textsc{m}+\frac{1}{2}-l)}{\cos\frac{\alpha}{2}(v_1+v_2+l) \cos\frac{\alpha}{2}(u_2+v_1-\frac{1}{2}+l)}} \nonumber
\end{align}
and
\begin{equation}
N_{\mathrm{t},0}=\sum_{k=0}^{\textsc{m}} \Delta_{\mathrm{t},k}= \sum_{j=0}^{\textsc{m}} \hat{\Delta}_{\mathrm{t},j} =
\prod_{1\leq l\leq\textsc{m}}  {\textstyle 
\frac{\sin\frac{\alpha}{2} (2u_1+l)\ \sin\frac{\alpha}{2} (u_1+u_2+v_1+v_2+l)  }{\sin\frac{\alpha}{2} (u_1-v_1-\frac{1}{2}+l)\sin\frac{\alpha}{2} (u_2+v_2-\frac{1}{2}+l)  } }.
\end{equation}
\end{subequations}

We thus conclude that the inverse  of the matrix
\begin{equation*}
  \mathbf{F}_{\mathrm{t}}=
 \begin{bmatrix}
f_{\mathrm{t}, 0}(\textsc{e}_{\mathrm{t}, 0}) & f_{\mathrm{t}, 0}(\textsc{e}_{\mathrm{t}, 1})&  \cdots & f_{\mathrm{t}, 0}(\textsc{e}_{\mathrm{t}, \textsc{m}}) \\
f_{\mathrm{t}, 1}(\textsc{e}_{\mathrm{t}, 0}) & f_{\mathrm{t}, 1}(\textsc{e}_{\mathrm{t}, 1})&  \cdots & f_{\mathrm{t}, 1}(\textsc{e}_{\mathrm{t}, \textsc{m}}) \\
\vdots &  \vdots &  \vdots  & \vdots \\
f_{\mathrm{t}, \textsc{m}-1}(\textsc{e}_{\mathrm{t}, 0}) & f_{\mathrm{t}, \textsc{m}-1}(\textsc{e}_{\mathrm{t}, 1})&  \cdots & f_{\mathrm{t}, \textsc{m}-1}(\textsc{e}_{\mathrm{t}, \textsc{m}}) \\
f_{\mathrm{t}, \textsc{m}}(\textsc{e}_{\mathrm{t}, 0}) & f_{\mathrm{t}, \textsc{m}}(\textsc{e}_{\mathrm{t}, 1})&  \cdots & f_{\mathrm{t}, \textsc{m}}(\textsc{e}_{\mathrm{t}, \textsc{m}}) 
\end{bmatrix} 
\end{equation*}
is given by (cf. Corollary \ref{Finv:cor})
\begin{subequations}
\begin{equation}
\mathbf{F}_{\mathrm{t}}^{-1}= N_{\mathrm{t},0}^{-1} \boldsymbol{\hat\Delta}_\mathrm{t}  \mathbf{F}_{\mathrm{t}}^T  \boldsymbol{\Delta}_\mathrm{t}  
\end{equation}
with
\begin{equation}
\boldsymbol{\Delta}_\mathrm{t}=\text{diag}\bigl( \Delta_{\mathrm{t},0},  \Delta_{\mathrm{t},1}  ,\ldots , \Delta_{\mathrm{t},\textsc{m}}         \bigr) , \quad
\boldsymbol{\hat\Delta}_\mathrm{t}=\text{diag}\bigl( \hat{\Delta}_{\mathrm{t},0},  \hat{\Delta}_{\mathrm{t},1}  ,\ldots , \hat{\Delta}_{\mathrm{t},\textsc{m}}         \bigr) ,
\end{equation}
while its determinant is given by
\begin{equation}
\det (\mathbf{F}_{\mathrm{t}}) = 
\frac{(-1)^{\frac{1}{2}\textsc{m}(\textsc{m}+1) } N_{\mathrm{t},0}^{\frac{1}{2}(\textsc{m}+1)} }{\sqrt{\prod_{0\leq l\leq\textsc{m}} \Delta_{\mathrm{t},l} \hat{\Delta}_{\mathrm{t},l}} }       .
\end{equation}
\end{subequations}

\section*{Acknowledgements}
The work of JFvD was supported in part by the {\em Fondo Nacional de Desarrollo
Cient\'{\i}fico y Tecnol\'ogico (FONDECYT)} Grant \# 1210015. TG was supported in part by the NKFIH Grant K134946.

\bigskip\noindent
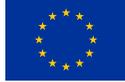
\parbox{.135\textwidth}{\begin{tikzpicture}[scale=.03]
\fill[fill={rgb,255:red,0;green,51;blue,153}] (-27,-18) rectangle (27,18);  
\pgfmathsetmacro\inr{tan(36)/cos(18)}
\foreach \i in {0,1,...,11} {
\begin{scope}[shift={(30*\i:12)}]
\fill[fill={rgb,255:red,255;green,204;blue,0}] (90:2)
\foreach \x in {0,1,...,4} { -- (90+72*\x:2) -- (126+72*\x:\inr) };
\end{scope}}
\end{tikzpicture}} \parbox{.85\textwidth}{This project has received funding from the European Union's Horizon 2020 research and innovation programme under the Marie Sk{\l}odowska-Curie grant agreement No 795471.}

\bibliographystyle{amsplain}

\end{document}